\documentclass[12pt]{amsart}
\usepackage{amsmath,amssymb}
\usepackage{imakeidx}
\usepackage[all]{xy}
\usepackage{tikz}
\usetikzlibrary{arrows,shapes,patterns}
\newtheorem{theorem}{Theorem}[section]
\newtheorem{lemma}[theorem]{Lemma}

\newtheorem{corollary}[theorem]{Corollary}
\newtheorem{prop}[theorem]{Proposition}
\newtheorem{observ}[theorem]{Observation}
\newcommand{\ideal}[1]{\mathfrak{#1}}
\newcommand{\union}{\Delta_1\cup\Delta_2}
\DeclareMathOperator{\core}{-core}
\DeclareMathOperator{\icore}{core}
\DeclareMathOperator{\ann}{ann}
\DeclareMathOperator{\spread}{-spread}

\makeindex
\begin{document}
\author{Thomas M. Ales}
\address{Department of Mathematics, George Mason University, Fairfax, VA 22030}
\email{tales@masonlive.gmu.edu}

\title[$*\core\ideal{m}$ in a Stanley Reisner ring]{
The $*\core$ of the graded maximal ideal in a Stanley-Reisner ring}
\maketitle
\begin{abstract}
We consider ideals $I$ in a Stanley-Reisner ring $k[\Delta]$ over the simplical complex $\Delta$, such that the tight closure of $I$, $I^*$, is equal to $\ideal{m}$, the standard graded maximal ideal of $k[\Delta]$. We determine the minimal number of generators of $I$ to be the $\dim \Delta+1$ and note the important role this value plays in bounding the intersection of all such ideals $I$. We make mention of this intersection in special cases of Stanley-Reisner rings. We conclude with a description of how this work relates to integral closure.
%The theory of tight closure is mostly viewed through a strictly algebraic lens.  The theory is applied to the much more combinatorial Stanley-Reisner rings in an attempt to use the structure of the simplicial complex to more quickly compute tight closure.  Special attention is paid to the $*$-reductions and the $*\core{\ideal{m}}$, the maximal ideal generated by the variables of the Stanley-Reisner ring, which are the ideals that tightly close to $\ideal{m}$ and the intersection of all such ideals respectively. Special types of Stanley-Reisner rings are examined and all notions are identified to be exactly equivalent to the analogous integral closure notions. 
\end{abstract}

\section{Introduction}
All rings in this paper are commutative, with identity, and Noetherian.  Further all rings also contain an infinite field of characteristic $p$ or characteristic 0, though making a distinction between the two has no bearing on the results of this paper, largely due to the work in \cite{HH-tcz}.

In section \ref{sr} we outline the basics of the relatively well known concept of a Stanley-Reisner ring, which were studied in depth by Hochster \cite{SR-Hoch} , Reisner \cite{SR-Reis}, and Stanley \cite{SR-Stan} in the 1970's. Specifically we discuss how a Stanley-Reisner ring relates to a simplicial complex, paying special attention to the dimension and minimal primes of a Stanley-Reisner ring.

We must also know a few key details, definitions, and theorems regarding tight closure.  These are covered in Section \ref{tc}. Tight closure was introduced in the 1980's for characteristic $p$ by Hochster and Huneke \cite{HHOrigin} \cite{HHmain} and later expanded to characteristic 0 \cite{HH-tcz}, and has since become a major topic of focus in commutative algebra. The most useful theorems in the section are Theorems \ref{poly} and \ref{inclusion} which enable us to use what we know about Stanley-Reisner rings from Section \ref{sr} and say a great deal about ideals that tightly close to the standard graded maximal ideal $\ideal{m}$ in a Stanley-Reisner ring.

For section \ref{reduc}, we pick a target ideal for the tight closure of an ideal $I$, namely the graded maximal ideal $\ideal{m}$ of a $k[\Delta]$, and describe the criteria required for $I$ to tightly close to that ideal.  Special mention is made of the number of generators and overall structure required of $I$.

In section \ref{bds}, we fix the number of generators of $I$ to its minimally required number and define $*\core$ to be the intersection of these minimally generated ideals.  We make mention of a known lower bound for $*\core{\ideal{m}}$, namely $\tau\cdot\ideal{m}$ where $\tau=\sum \ann{P}$ for all minimal primes $P$ of $k[\Delta]$ and associated pitfalls.  We also introduce an upper bound that is consistent for all Stanley-Reisner rings and a lower bound that depends only on the dimension of the ring (and the dimension of the complex). 

We then proceed with a discussion of $*\core{\ideal{m}}$ in special cases with an eye toward determining $*\core{\ideal{m}}$ when $\Delta$ is a simple graph and reduce all future work to the question of what $\core{\ideal{m}}$ is in the case where $\Delta$ is connected.

Finally, in section \ref{ic} we discuss the broader implications of the work in the paper to integral closure.  Specifically we notice that everything proven about tight closure in this paper is also true for integral closure, and the paper may therefore be viewed in the context of integral closure if it is more relevant to the reader.
\section{Stanley-Reisner rings}
\label{sr}

Let $V=\{v_1,\ldots, v_n\}$ be a finite set. A \textit{simplicial complex} $\Delta$ on $V$ is a collection of elements from $2^{V}$, the power set of $V$, such that if $F\in \Delta$ and $G\subset F$, then $G\in\Delta$, and such that $\{v_i\}\in\Delta$ for $i=1,\ldots,n$. The elements of $\Delta$ are called \textit{faces} and the maximal faces (under inclusion) are called \textit{facets}. The dimension of a face $F$ of $\Delta$ is $\dim F=|F|-1$, and the dimension of a simplicial complex is \[\dim\Delta=\max\{\dim F: F \in\Delta\}.\] We define a simplicial complex $\Delta$ to be \textit{proper} if $\Delta\neq 2^V$, otherwise $\Delta$ is a simplex. Further for $d<n$, we define $\Delta_{d,n}$ be the $d-1$ dimensional proper simplicial complex on $n$ vertices such that every facet of $\Delta$ is a set of size $d$ and all size $d$ subsets of $V$ are facets.  The distinction that $d$ is strictly less than $n$ is important because if $d=n$, $\Delta_{d,n}$ is a simplex and $d>n$ is impossible.
 
 If $\Delta$ is a simplicial complex over vertex set $V=\{v_1,\ldots, v_n\}$, define $I_{\Delta}$ to be the ideal of $k[x_1,\ldots,x_n]$, the polynomial ring over an infinite field $k$ of characteristic $p$ or $0$, generated by all monomials $x_{i_{1}}x_{i_{2}}\cdots x_{i_{s}}$ such that $\{v_{i_{1}},v_{i_{2}},\ldots,v_{i_{s}}\}\notin \Delta$.We define the Stanley-Reisner ring $k[\Delta]$ to be the quotient ring $k[x_1,\ldots,x_n]/I_{\Delta}$. The ring $k[\Delta]$ is a polynomial ring if and only if $\Delta$ is a simplex.  There exists a natural one-to-one correspondence between all simplicial complexes on $n$ vertices and all square free monomial ideals of $k[x_1,\ldots,x_n]$ with generators of degree greater than 1.
 
 Stanley-Reisner rings have a few useful and interesting properties. The first of these can  be found in \cite{BH}, among other sources, and provides a characterization of all the minimal primes of the ring $k[\Delta]$ as well as the dimension of the ring:
 \begin{theorem}[\cite{BH}, Theorem 5.1.4]
 Let $\Delta$ be a simplicial complex and $k$ a field then\[I_{\Delta}=\bigcap\limits_{F}\mathfrak{B}_{F}\]where the intersection is taken over all facets $F$ of $\Delta$, and $\mathfrak{B}_F$ denotes the prime ideal generated by all $x_i$ such that $v_i\notin F$.  In particular, \[\dim k[\Delta]=\dim\Delta+1.\]
 \end{theorem} 
 For example, the ring $k[\Delta]$ with $I_{\Delta}=(x_1x_4,x_2x_4)$ corresponds to the simplicial complex $\Delta$ in the following figure with facets $\{v_1,v_2,v_3\}$ and $\{v_3,v_4\}$. 
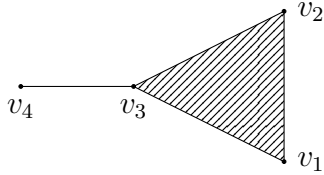
 \[\begin{tikzpicture}
[scale=2, vertices/.style={draw, fill=black, circle, inner sep=0.5pt}]
\node[vertices, label=below:{$v_4$}] (w) at (0,0){};
\node[vertices, label=below:{$v_3$}] (x) at (0.75,0) {};
\node[vertices, label=right:{$v_2$}] (z) at (1.75,0.5) {};
\node[vertices, label=right:{$v_1$}] (y) at (1.75,-0.5) {};

\draw[pattern=north east lines] (x.center) -- (y.center) -- (z.center) -- cycle;
\draw (w.center) -- (x.center);
\end{tikzpicture}\] The minimal primes of $k[\Delta]$ corresponding to each of these facets are, respectively, $P=(x_4)$ and $Q=(x_1,x_2)$.  The dimension of this ring is three because $\dim\Delta=2$ and $P\cap Q=(x_1x_4,x_2x_4)=I_{\Delta}$.

\section{Tight Closure}
\label{tc}

Let $I\subset R$ be an ideal, where $R$ is a ring of prime characteristic $p>0$. The \textit{tight closure} $I^{*}$ of $I$ is the set of all elements $x\in R$ for which there exists $c\in R^{\circ}$ with $cx^{p^e}\in I^{[p^e]}$ for $p^e\gg 0$, where $R^{\circ}$ is the set of elements of $R$ not contained in any minimal prime of $R$.  One says $I$ is \textit{tightly closed} if $I=I^{*}$. When $R$ is a finitely generated algebra over a field $k$ of characteristic 0, one passes to characteristic $p$ models of $R$ and $x\in I^*$ if true for almost all characteristic $p$ models. For details see \cite{HH-tcz}.

Since for a minimal prime $P$ in a Stanley-Reisner ring $k[\Delta]$, $k[\Delta]/P$ is a polynomial ring, the following result from \cite[Theorem 4.4]{HHmain} in characteristic $p$ and \cite[Theorem 4.1.1]{HH-tcz} for characteristic 0 will be useful.
\begin{theorem}
\label{poly}
If $R$ is a polynomial ring over a field $k$ and $I$ is an ideal of $R$, then $I^*=I$.
\end{theorem}

If we pair Theorem \ref{poly} with Theorem \ref{inclusion}, we can compute the tight closure of ideals in Stanley-Reisner rings relatively easily.  Theorem \ref{inclusion} is found as \cite[Lemma 2.10(c)(1)]{AHH} for characteristic $p$ and \cite[Theorem 2.5.5(n)]{HH-tcz} for characteristic 0.
\begin{theorem}
\label{inclusion}
Let $R$ be a Noetherian ring, $I$ an ideal of $R$, and $x\in R$.  Then $x\in I^{*}$ if and only if for all minimal primes $P$ of $R$, the image of $x$ is in $(IR/P)^*$ as an ideal of $R/P$.
\end{theorem}

As an example of how Theorems \ref{poly} and \ref{inclusion} can work together, let $k[\Delta]$ be the Stanley-Reisner ring corresponding to the simplicial complex consisting of only two vertices:\[\begin{tikzpicture}
[scale=2, vertices/.style={draw, fill=black, circle, inner sep=0.5pt}]
\node[vertices, label=below:{$x$}] (w) at (0,0){};
\node[vertices, label=below:{$y$}] (x) at (0.75,0.75) {};

\end{tikzpicture}\]\[k[\Delta]=k[x,y]/(xy).\] The set of minimal primes is $\{(x), (y)\}$.  Let $I=(x+y)$ be an ideal of $R$.  Then, \[Ik[\Delta]/(x)=yk[y]=(yk[y])^{*}\]\[Ik[\Delta]/(y)=xk[x]=(xk[x])^{*}\] by Theorem \ref{poly}. In $Ik[\Delta]/(x)$, $\overline{x}=\overline{0}$ and in $Ik[\Delta]/(y)$, $\overline{x}=\overline{x+y}$, so by Theorem \ref{inclusion}, $x\in (x+y)^{*}$.  Similarly, $y\in (x+y)^*$.  Thus, \[(x,y)\subseteq (x+y)^{*}\subseteq (x,y)^{*}=(x,y)\] and the tight closure of $(x+y)$ is $(x,y)$.

Because if $\Delta\subseteq\Delta'$ are simplicial complexes, there exists a natural surjection $f:k[\Delta']\twoheadrightarrow k[\Delta]$. It will be important to know how the tight closure of an ideal is affected by such a map.  For this we turn to \cite[Theorem 6.24]{HHbase} if the base field $k$ is characteristic $p$ and to \cite[Theorem 2.5.5(k)]{HH-tcz} if $k$ is characteristic 0.
\begin{theorem}
\label{persist}
If $f: R \rightarrow S$ is a map of finitely generated algebras over a field, $I$ an ideal of $R$, and $x\in I^*$, then $f(x) \in (f(I)S)^*$. In particular $f(I^*)\subseteq f(I)^*$.
\end{theorem}

As a corollary to Theorem \ref{persist}, we get the following result which will be useful in Section \ref{bds}:

\begin{corollary}
\label{persistcor}
When $R$ is finitely generated over a field $k$, $f:R\rightarrow S$ a ring isomorphism, and $I$ an ideal of $R$, then $f(I^{*})=f(I)^*$.
\end{corollary}
\begin{proof}
By Theorem \ref{persist}, $f(I^*)\subseteq f(I)^*$. Let $g:S\rightarrow R$ be the inverse of $f$. But since $S$ is also finally generated over a field, by Theorem \ref{persist} again, we get $g(f(I)^*)\subseteq g(f(I))^*=I^*$.  Apply $f$ to both sides to get $f(g(f(I^*)))\subseteq f(I^*)$. Thus $f(I)^*\subseteq f(I^*)$ which gives $f(I^*)=f(I)^*$.
\end{proof}

We also need the following simple observation:

\begin{observ}
Let $P$ be a prime ideal of $R$. Then \[P^*=P.\]
\end{observ}
\begin{proof}
The tight closure of an ideal $I$ is contained in integral closure of $I$ \cite{HHmain}. All prime ideals are integrally closed by \cite[Remark 1.1.3 (4)]{HuSw-book}. If $P^-$ represents the integral closure of a prime ideal $P$, then \[P\subseteq P^*\subseteq P^-=P.\]Thus $P^*=P$.
\end{proof}

\section{$*$-reductions and $*$-spread of $\ideal{m}$}
\label{reduc}

Given ideals $J\subseteq I$, we say $J$ is a $*$-\textit{reduction} of $I$ if $I\subseteq J^*$ (equivalently $I^*=J^*$). A $*$-reduction $J$ is \textit{minimal} if for all ideals $K\subsetneq J$, $I\not\subset K^{*}$. We define the $*$\textit{-spread} of an ideal $I$ to be fewest number of generators of a minimal $*$-reduction of $I$.

If we let $\ideal{m}=(x_1,\ldots,x_n)$ in the Stanley-Reisner ring $k[\Delta]=k[x_1,\ldots,x_n]/I_{\Delta}$ we can find effective descriptions of the minimal $*$-reductions of $\ideal{m}$ as well as state exactly the $*$-spread of $\ideal{m}$.

When determining what a $*$-reduction of $\ideal{m}$ must look like, we can build the ideal by examining what conditions the ideal must meet.  It must be true that for a prime $P$ of $k[\Delta]$ and a $*$-reduction $I$ with $s$ of $\ideal{m}$ with $s$ generators that $I+P/P=\ideal{m}$.  So whatever the generators of $I$, we know the linear parts of the generators remaining in $I+P/P$ must form a system of linear equations in the variables of $k[\Delta]$ that are outside of $P$. When considering all minimal primes of a Stanley-Reisner ring in $n$ variables, the coefficients of the linear terms form an $s\times n$ matrix of values representing a linear system such that for each specific minimal prime $P$, the columns of the matrix that represent a variable outside of $P$ must on their own form a matrix that can be converted to reduced row echelon form with a leading entry in each column.

For an example consider the ring $\mathbb{R}[x,y,z]/(xz)$ and the ideal $I=(x+y+2z,x+2y+z)$. It is not to hard to check that $I$ is a $*$-reduction of $\ideal{m}$.  The minimal primes of this ring are $P=(x)$ and $Q=(z)$.  If we focus on $P$ we get a matrix that reduces as follows:\[\begin{pmatrix}
1 & 2\\
2 & 1
\end{pmatrix}
\rightarrow
\begin{pmatrix}
1 & 2\\
0 & -3
\end{pmatrix}
\rightarrow
\begin{pmatrix}
1 & 2\\
0 & 1
\end{pmatrix}
\rightarrow
\begin{pmatrix}
1 & 0\\
0 & 1
\end{pmatrix}\]where the first column is for the variable $y$ and the second column is for $z$. Similarly, for $Q$ the matrix reduces as such\[\begin{pmatrix}
1 & 1\\
1 & 2
\end{pmatrix}
\rightarrow
\begin{pmatrix}
1 & 1\\
0 & 1
\end{pmatrix}
\rightarrow
\begin{pmatrix}
1 & 0\\
0 & 1
\end{pmatrix}\] for $x$ in the first column and $y$ in the second column.

One of the things we will do often is to apply these same row reductions the matrix of coefficients of the linear terms without reducing the number of columns. This will provide alternate generating sets for the same ideal $I$. If we continue with the previous example, for the ideal $P$, \[\begin{pmatrix}
1 & 1 & 2\\
1 & 2 & 1
\end{pmatrix}
\rightarrow
\begin{pmatrix}
1 & 1 & 2\\
-1 & 0 & -3
\end{pmatrix}
\rightarrow
\begin{pmatrix}
1 & 1 & 2\\
\frac{1}{3} & 0 & 1
\end{pmatrix}
\rightarrow
\begin{pmatrix}
\frac{1}{3} & 1 & 0\\
\frac{1}{3} & 0 & 1
\end{pmatrix}
\]and for $Q$ \[\begin{pmatrix}
1 & 1 & 2\\
1 & 2 & 1
\end{pmatrix}
\rightarrow
\begin{pmatrix}
1 & 1 & 2\\
0 & 1 & -1
\end{pmatrix}
\rightarrow
\begin{pmatrix}
1 & 0 & 3\\
0 & 1 & -1
\end{pmatrix}\]which provides two alternate ways to generated $I$, specifically $(\frac{x}{3}+y,\frac{x}{3}+z)$ and $(x+3z,y-z)$ respectively. 

For the reasons outlined in the example, we will often choose to generate a $*$-reduction $I=(f_1,\ldots,f_s)$ of $\ideal{m}$ ,with respect to a chosen prime $P$, by the set of polynomials $\{x_1+g_1, \ldots, x_r+g_r, g_{r+1},\ldots, g_s\}$ where (after relabeling) $x_1,\ldots,x_r$ are the variables outside of $P$ and $g_i$ for $1\leq i\leq s$ are polynomials that exist in $P$. We will refer to this process as \textit{diagonalization}.

Before we examine the generators of the $*$-reductions of $\ideal{m}$ any farther, let us first determine the number of generators a minimally generated $*$-reduction must have.

\begin{observ}
\label{observ}
The $*\spread$ of $\ideal{m}$ is at least $d=\dim k[\Delta]$.
\end{observ}
\begin{proof}
Let $I$ be a $*$-reduction of $\ideal{m}$ in $k[\Delta]$ and let $d=\dim k[\Delta]$.  Suppose $P$ is a prime of $k[\Delta]$ such that $I\subseteq P$.  Then $\ideal{m}=I^{*}\subseteq P^{*}=P$.   Since $\ideal{m}$ is maximal, $\ideal{m}=P$. Thus $\ideal{m}$ is the only prime containing $I$.  The height of $I$ is therefore $d$.  By Krull's height theorem \cite[Theorem A.2.]{BH}, $I$ has at least $d$ generators.
\end{proof}
\begin{theorem}
\label{ideal}
The $*\spread$ of $\ideal{m}$ in $k[\Delta_{d,n}]$ is $d$.
\end{theorem}

\begin{proof}
For almost every choice of coefficients $a_{(i,j)}\in k$, we want to show that $I=(f_1,\ldots, f_d)$ with 
\begin{align*}
f_1 &= a_{(1,1)}x_1+\cdots +a_{(1,n)}x_n\\
f_2 &= a_{(2,1)}x_1+\cdots +a_{(2,n)}x_n\\
&\ \ \vdots\\
f_d &= a_{(d,1)}x_1+\cdots +a_{(d,n)}x_n
\end{align*}
satisfies $I^*=\ideal{m}$.  Let $X$ be the set of variables of $k[\Delta_{d,n}]$. The ideal $I$ is a $*$-reduction of $\ideal{m}$ if $IR/P \cong (x_{i_1},\ldots,x_{i_d})$ in the ring $k[x_{i_1},\ldots, x_{i_d}]$ for any minimal prime $P=(X-\{x_{i_1},\ldots, x_{i_d}\})$ of $R$.  Therefore if we have $d$ generators of $I$, for any choice of $d$ of the variables, we get a square matrix with the first column corresponding to the coefficients on $x_{i_1}$ across the $d$ linear polynomial generators of $I$.  Due to our discussion of diagonalization, the desired condition is equivalent to all such matrices being nonsingular.
 
Let $\gamma_{j_{1},\ldots,j_{d}}$ be the determinant of the matrix of coefficients for columns $j_1,\ldots,j_{d}$.  These are all nonzero if and only if \[\gamma=\prod\gamma_{j_1,\ldots,j_d}\neq 0.\]The product $\gamma$ is a polynomial in $n\cdot d$ variables and its solution cuts out a hypersurface in affine $n\times d$ space.  Therefore by \cite[Proposition 1.3.]{Kunz-book}, almost any choice of values for $a_{(i,j)}$ will give us the desired ideal $I$, so that in particular, such $I$ do exist.
\end{proof}

Interestingly, we can use the ideals constructed in Theorem \ref{ideal} to find the $*$-spread of $\ideal{m}$ in an arbitrary Stanley-Reisner ring:

\begin{theorem}
\label{spreaddn}
Let $\Delta$ be a proper $d-1$ dimensional simplicial complex on $n$ vertices. Then $*$-spread of $\ideal{m}$ in $k[\Delta]$ is $d$. That is, \[*\spread\ideal{m}=\dim k[\Delta]=\dim\Delta+1.\]
\end{theorem}

\begin{proof}
Note that $\Delta\subset \Delta_{d,n}$. Thus there exists a natural surjection \[\varphi:k[\Delta_{d,n}]\twoheadrightarrow k[\Delta].\] If we define $I$ to be as in Theorem \ref{ideal}, then $\varphi(I)$ is an ideal of $k[\Delta]$.  If $\ideal{n}$ is the maximal ideal of $k[\Delta_{d,n}]$ generated by the images of the variables, then \[\ideal{m}=\varphi(\ideal{n})=\varphi(I^*)\subseteq \varphi(I)^*\subseteq \ideal{m},\]where the first inclusion follows from Theorem \ref{persist}.  Hence $\varphi(I)^*=\ideal{m}$ which means $\ideal{m}$ has a $*$-reduction with $d$ generators.  Thus $*\spread{\ideal{m}}\leq d$.  But by Observation \ref{observ} $*\spread{\ideal{m}}\geq d$ as well.  Hence $*\spread{\ideal{m}}=d$.
\end{proof}
In the previous theorems, the $*$-reductions of $\ideal{m}$ used in the proofs were generated by linear polynomials.  It is important to note that a $*$-reduction of $\ideal{m}$, even a minimal one, may be generated by polynomials that not are linear.  For example, the ideal $J=(x+y+xz,y+z)$ in $k[x,y,z]/(xyz)$ is a $*$-reduction of $\ideal{m}$ and is not generated by linear polynomials.  If we let $I$ be an ideal in the same ring generated by the linear parts of the generators of $I$, i.e. $I=(x+y,y+z)$, we see that $I$ is also a $*$-reduction of $\ideal{m}$.  This indicates the following result.
\begin{theorem}
Let $J=(f_1,\ldots,f_s)$ be a $*$-reduction of $\ideal{m}$ in $k[\Delta]$ for $s\geq d=\dim k[\Delta]$. Let $f_i=g_i+h_i$ for $1\leq i\leq s$ where $g_i$ is the polynomial of linear summands of $f_i$ and $h_i$ is the polynomial of nonlinear summands of $f_i$.  If $I=(g_1,\ldots, g_s)$, then $I^*=\ideal{m}$.
\end{theorem}
\begin{proof}
Let $P$ be a minimal prime of $k[\Delta]$. Then $J+P/P=\ideal{m}/P$.  Let $\overline{f}_i$ be the image of $f_i$ in $J+P/P$ and let $\overline{x}_j$ be the image of $x_j$ in $\ideal{m}/P$. Then there exist polynomials $a_1, \ldots, a_s$ in $k[\Delta]/P$ such that \[\overline{x}_j=a_1\overline{f}_1+\cdots+a_s\overline{f}_s\] in $J+P/P$. Let each $a_i=c_i+b_i$ where $c_i$ is the constant term of $a_i$ and $b_i$ is the sum of every other term. Then \[\overline{x}_j=\sum c_i\overline{g}_i+\sum b_i\overline{g}_i+\sum a_i\overline{h}_i\] and since all the terms of $\sum b_i\overline{g}_i+\sum a_i\overline{h}_i$ are of degree greater than one, it must be true that $\sum b_i\overline{g}_i+\sum a_i\overline{h}_i=0$.  Thus for all minimal primes $P$ of $k[\Delta]$, \[I+P/P=J+P/P=\ideal{m}/P\] and $I^*=\ideal{m}$.
\end{proof}

Even though the $*$-spread of $\ideal{m}$ is $d$, there do exist $*$-reductions with more than $d$ generators.  Such a reduction, however, is never minimal. Using iterations of the following theorem we can show that given a $*$-reduction $J$ of $\ideal{m}$ with more than $d$ generators, we can find an ideal $I\subset J$ that is a $*$-reduction of $\ideal{m}$ with $d$ generators. The following result is analogous to a result of Epstein \cite[Theorem 5.1]{nme*spread}, though his result his result is for excellent analytically irreducible local domains of characteristic $p>0$ and the method of proof is different.
\begin{theorem}
Let $d=\dim k[\Delta]$ and let $J=(f_1,\ldots,f_c)$ with $c\geq d+1$. If $J^*=\ideal{m}$, then there exists $I=(g_1,\ldots,g_{c-1})$ such that $I\subset J$ and $I^*=\ideal{m}$. 
\end{theorem}

\begin{proof}
Let $P_1,\ldots,P_s$ be the minimal primes of $k[\Delta]$ ordered such that for $r\leq s$, if $P_i\in\{P_1,\ldots,P_r\}$ and $J$ as defined above, \[(f_1,\ldots,f_{c-1})k[\Delta]/P_i\equiv J k[\Delta]/P_i \equiv  \ideal{m}k[\Delta]/P_i\] and if $P_i\in\{P_{r+1},\ldots,P_s\}$, the above is not true.  Let $Q=(x_{t+1},\ldots, x_n)\in\{P_{r+1},\ldots,P_s\}$. Then there exists $f_1',\ldots,f_{c}'\in k[\Delta]$ such that the linear part of each of these polynomials only include variables from $\{x_{t+1},\ldots,x_n\}$ such that \[J=(x_1+a_1x_t+f_1',\ldots, x_{t-1}+a_{t-1}x_t+f_{t-1}',f_t',\ldots, f_{c-1}',x_t+f_c')\] and \[(f_1,\ldots,f_{c-1})=(x_1+a_1x_t+f_1',\ldots, x_{t-1}+a_{t-1}x_t+f_{t-1}',f_t',\ldots, f_{c-1}')\] for $a_1,\ldots,a_t\in k$. Then setting $g_i=a_ix_t+f_i',1\leq i\leq t-1$, we have \[(x_1+g_1,\ldots, x_{t-1}+g_{t-1},\alpha (x_t+f_c')+f_t',\ldots, f_{c-1}')k[\Delta]/Q\equiv \ideal{m}k[\Delta]/Q\] for any nonzero $\alpha\in k$.

We want to show that we can choose the above $\alpha$ in such a way that with $J'=(x_1+g_1,\ldots, x_{t-1}+g_{t-1},\alpha (x_t+f_c')+f_t',\ldots, f_{c-1}')$, \[J'k[\Delta]/P_i\cong \ideal{m}k[\Delta]/P_i\] for all $P_i\in\{P_1,\ldots,P_r\}.$

For each $P_i\in\{P_1,\ldots, P_r\}$ attempt to diagonalize the generators of $J'$ as in the beginning of the section except for $\alpha(x_t+f_c')+f_t'$, leave that generator untouched. The diagonalization of the other $c-2$ generators will include a leading term in the column for each variable outside of $P_i$ with the exception of at most one of the variables.  If every necessary column has a leading term after the diagonalization process, then \[J'k[\Delta]/P_i\cong \ideal{m}k[\Delta]/P_i\]no matter the choice of $\alpha$. Otherwise, the diagonalization misses exactly one column. In this case we need the generator $\alpha(x_t+f_c')+f_t'$ to accommodate this column. This depends on the choice of $\alpha$.

Assume that the column this generater is needed to accommodate is the one for the variable $y$ and that $\beta$ is the coefficient on $y$ and that $\alpha_j$ is the coefficient on $x_j$ in  $\alpha(x_t+f_c')+f_t'$.  Part of the diagonalization process includes performing a row operation to remove $\alpha_jx_j$ as part of this sum, which potentially alters the coefficients on the terms of $\alpha(x_t+f_c')+f_t'$, including $\beta$. It is therefore necessary for $\beta$ to not be both equal in magnitude and opposite in sign to the cumulative effect of these row operations i.e. for each $\gamma_j$, the coefficient on $y$ in the generator with leading term in the $x_j$ column, $0\neq \beta-\sum\alpha_j\gamma_j$. Since $\alpha$ only appears once in this equation as part of the construction of $\beta$, only at most one value of $\alpha$ will not work for each $P_i$. Therefore since there are $r$ minimal primes $P_i$ that we need to check, there are at most $r$ values of $\alpha$ that will not work.

Repeat this process from the beginning of the proof until $r=s$ and let final resulting $J'=I$.  Thus there exists an ideal $I\subset J$ with one less generator than $J$ such that $I^*=\ideal{m}$.
\end{proof}
\begin{corollary}
\label{subideal}
If $I$ is a $*$-reduction of $\ideal{m}$ in $k[\Delta]$ with more than $d=\dim k[\Delta]$ generators, then there exists an ideal $J\subset I$ of $k[\Delta]$ with $d$ generators such that $J^*=\ideal{m}$.
\end{corollary}
\begin{proof}
This follows naturally by reverse induction on $c$.
\end{proof}
\section{Bounds for $*\core{\ideal{m}}$}
\label{bds}

When considering integral closure ideal, there is a notion of the core of an ideal $I$, which is the intersection of all minimal reductions $I$.  We can define an analogous notion for tight closure.  We define the $*\core$ of an ideal $I$ to be the intersection of all minimal $*$-reductions of $I$.  If we return to the ring $k[\Delta]=k[x,y]/(xy)$, we saw that $(x+y)$ was a (minimal) $*$-reduction of $\ideal{m}$.  In fact, all minimal $*$-reductions of $\ideal{m}$ in this ring are of the form $(x+\lambda y)$, where $\lambda$ is any nonzero element of $k$.  Then, \[*\core{\ideal{m}}=\bigcap\limits_{\lambda} (x+\lambda y)=(x^2,y^2).\]

The importance of using an infinite field $k$ in the above example cannot be overstated.  For example, if $R=k[\Delta_{1,2}]=k[x,y]/(xy)$ where $k=\mathbb{F}_2$, then the only minimal $*$-reduction of $\ideal{m}$ is $(x+y)$, which means that the intersection of all minimal $*$-reductions of $\ideal{m}$ is $(x+y)$. The upper and lower bound we introduce in this section are given with the intention that $k$ is infinite. For this example, we notice $*\core\ideal{m}$ is generated by a pair of monomials.  In general, we can observe the following:
\begin{observ}
Suppose $I\subseteq k[\Delta]$ is an ideal generated by monomials and that $I=I^{*}$. Then $*\core{I}$ is generated by monomials.
\end{observ}
\begin{proof}
Let $S=k[x_1,\ldots,x_n]$, a polynomial ring in $n$ variables over an infinite field $k$ and let $I_{\Delta}$ be the defining ideal of a simplicial complex such that $k[\Delta]=S/I_{\Delta}$. There exists a group action of $G=(k^{\times})^n$ on $S$ defined by \[(\lambda_1,\ldots,\lambda_n)\cdot f(x_1,\ldots,x_n):=f(\lambda_1x_1,\ldots,\lambda_nx_n).\]The fixed ideals of $S$ under $G$ are the monomial ideals.  The action of $G$ on $S$ induces an action on $k[\Delta]$. Let $J$ be an ideal of $k[\Delta]$ such that $J^*=I$.  Then for any $g\in G$, $g\cdot J=\{g\cdot a\ |\ a\in J\}$ is an ideal and $(g\cdot J)^*=I$ by Corollary \ref{persistcor}. Thus \[*\core I=\bigcap_{\{J|J^*=I\}}J=\bigcap_{\{J|J^*=I\}}g\cdot J=g\cdot\bigcap_{\{J|J^*=I\}}J=g\cdot *\core I.\]Since the group action fixes $*\core I$, $*\core I$ must be generated by monomials.
\end{proof}
%Since, by supposition, $k$ is infinite, every $*\core{\ideal{m}}$ calculation will be an intersection of infinitely many ideals.  This makes using a program like Macaulay2 problematic.  The first thing we will do is somewhat limit the ideals we must look at to calculate $*\core{\ideal{m}}$.

%\begin{lemma}
%In any Stanley-Reisner ring $k[\Delta]$, $*\core{\ideal{m}}$ can be computed using only the minimal $*$-reductions of $\ideal{m}$ that are generated by linear polynomials.
%\end{lemma}

%\begin{proof}
%Linear part of generators is $*$-reduction.

%something about non linear terms disappear quickly
%\end{proof}

Calculating the $*\core$ of any ideal can be difficult, so we want to have a better idea of where we should look for the monomials that generate the $*\core$. The first place we look is to the test ideal $\tau$ of the ring $R$.  The test ideal $\tau$ is defined the following way:\[\tau:=\bigcap_{I\text{ ideal of }R}(I:I^{*})\] which by Vassilev \cite[Theorem 3.7]{Va-testquot} can also be defined to be the sum of the annihilating ideals of the minimal primes in a Stanley-Reisner ring $k[\Delta]$.  The test ideal is not hard to find in the setting of Stanley-Reisner rings, especially since the minimal primes of a Stanley-Reisner ring $k[\Delta]$ are easy to describe.  In our example of $k[\Delta]=k[x,y]/(xy)$, the annihilator of $(x)$ is $(y)$ and the annihilator of $(y)$ is $(x)$, therefore $\tau=(x,y)$.

The following observation about the test ideal $\tau$ from \cite[Observation 3.1]{FoVaVr-*core} provides a computationally based lower bound for the $*\core$ of an ideal.

\begin{observ}
Let $R$ be a ring of any characteristic with test ideal $\tau$.  Let $I$ be an ideal of $R$.  Then $\tau I\subseteq *\core{I}$.
\end{observ}

In simple cases, $\tau I$ is exactly the $*\core{I}$.  For example, in $k[x,y]/(xy)$, $\tau \ideal{m}=(x,y)^2=(x^2,y^2)=*\core\ideal{m}$. As the dimension of Stanley-Reisner rings increases, this lower bound does not capture all the information about $*\core{I}$.  The ring $k[x,y,z]/(xy)$ has test ideal $\tau=(x,y)$ and \[\tau\ideal{m}=(x,y)\cdot(x,y,z)=(x^2,xz,y^2,yz),\] but the monomial $z^2$ is also easily computed to be in $*\core{\ideal{m}}$. The difficulty of computation also tends to increase as simplicial complexes get more complicated.  Because of this, we provide a simpler lower bound for $*\core{\ideal{m}}$, namely $\ideal{m}^{d+1}\subseteq *\core{\ideal{m}}$.  Often, this lower bound is exactly $*\core{\ideal{m}}$.  To show this we need the following lemma.

\begin{lemma}
\label{lemma}
Let $\Delta_{d,n}$ be the complete $d-1$ dimensional simplicial complex on $n$ vertices. Let $\alpha_1,\alpha_2,\ldots,\alpha_s$, $s\leq d$, be a partition of the positive integer $d+1$ with \[\alpha_1\geq\alpha_2\geq\cdots\geq\alpha_s\geq1;\alpha_1\geq 2\]Then all monomials of the form $x_{i_1}^{\alpha_{1}}x_{i_2}^{\alpha_{2}}\cdots x_{i_s}^{\alpha_s}$ (with $i_1,\ldots,i_s$ distinct) are in the $*\core{\ideal{m}}$ if all monomials of the form $x_{j_1}^{\alpha_{1}-1}x_{j_2}^{\alpha_{2}}\cdots x_{j_s}^{\alpha_s}x_{j_{s+1}}$ (with $j_1,\ldots,j_{s+1}$ distinct) are in $*\core{\ideal{m}}$.
\end{lemma}

\begin{proof}
Since all simplicial complexes of the form $\Delta_{d,n}$ are uniform, without loss of generality, if we can show that a single monomial of the form $x_{i_1}^{\alpha_{1}}x_{i_2}^{\alpha_{2}}\cdots x_{i_s}^{\alpha_s}$ is in $*$-core($\ideal{m}$), then we have shown they all are in $*$-core($\ideal{m}$).  We therefore will let $i_1=1$, $i_2=2$,$\ldots, i_s=s$.  For any minimal $*$-reduction of $I=(f_1,\ldots,f_d)$ of $\ideal{m}$, we know by Theorem \ref{ideal} and Corollary \ref{subideal} that $I=(x_1+g_1,\ldots,x_d+g_d)$ where each $g_t$, $1\leq t\leq d$ is a polynomial with no linear terms in the variables $x_{1},\ldots, x_d$.  Multiply $x_1+g_1$ by $x_{1}^{\alpha_{1}-1}x_{2}^{\alpha_{2}}\cdots x_{s}^{\alpha_s}$.  Then since $d+1> s$, $x_{1}^{\alpha_{1}-1}x_{2}^{\alpha_{2}}\cdots x_{s}^{\alpha_s}\cdot g_1$ is a polynomial in which all the terms are divisible by a monomial of the form $x_{1}^{\alpha_{1}-1}x_{2}^{\alpha_{2}}\cdots x_{s}^{\alpha_s}x_{j_{s+1}}$ for $d+1\leq j_{s+1}\leq n$.  Since $x_1+g_1\in I$ and each term of $x_{1}^{\alpha_{1}-1}x_{2}^{\alpha_{2}}\cdots x_{s}^{\alpha_s}\cdot g_1$ is in $*$-core($\ideal{m}$)$\subseteq I$, $x_{1}^{\alpha_{1}}x_{2}^{\alpha_{2}}\cdots x_{s}^{\alpha_s}$ must also be in $I$.  Since $I$ was arbitrary, $x_{1}^{\alpha_{1}}x_{2}^{\alpha_{2}}\cdots x_{s}^{\alpha_s}$ is in every minimal $*$-reduction of $\ideal{m}$ and is therefore in $*$-core($\ideal{m}$).
\end{proof}

\begin{theorem}
\label{lower}
Let $\Delta_{d,n}$ be the complete $d-1$ dimensional simplicial complex on $n$ vertices, and $\ideal{m}$ the maximal ideal of $k[\Delta_{d,n}]$.  Then \[\ideal{m}^{d+1}\subseteq *\core{\ideal{m}}.\]
\end{theorem}

\begin{proof}
The ideal $\ideal{m}^{d+1}$ is generated by all monomials of degree $d+1$ over the variables $x_1,\ldots,x_n$.  If we can show all such monomials of are in $*\core{\ideal{m}}$, then $*\core{\ideal{m}}$ is bounded below by $\ideal{m}^{d+1}$.  The degree distribution on any monomial corresponds to a partition of $d+1$.  The length of a partition, $s$, is $1\leq s\leq d+1$. In $k[\Delta_{d,n}]$, any product of $d+1$ distinct variables is 0.  Therefore, all products of $d+1$ distinct variables are in $*\core{\ideal{m}}$.  These correspond to the only partition of length $d+1$.

Suppose that all monomials corresponding to partitions of $d+1$ of length $s$ are in $*\core{\ideal{m}}$, $s>1$.  Let $x_{i_1}^{\alpha_1}\cdots x_{i_{s-1}}^{\alpha_{s-1}}$ be a monomial corresponding to a partition of length $s-1$ with $\alpha_1\geq\alpha_2\geq\ldots\geq\alpha_{s}\geq 1$.  The exponent $\alpha_1\geq 2$ by the pigeonhole principle. Then by Lemma \ref{lemma}, $x_{i_1}^{\alpha_1}\cdots x_{i_{s-1}}^{\alpha_{s-1}}\in *\core{\ideal{m}}$ because all monomials of the form $x_{j_1}^{\alpha_1-1}\cdots x_{j_{s-1}}^{\alpha_{s-1}}x_{j_s}\in *\core{\ideal{m}}$.  Thus by induction, $\ideal{m}^{d+1}\subseteq *\core{\ideal{m}}$.
\end{proof}

\begin{theorem}
Let $\Delta_{d,n}$ be the complete $d-1$ dimensional simplical complex on $n$ vertices, and $\ideal{m}$ the maximal ideal of $k[\Delta_{d,n}]$ generated by the variables.  Then\[*\core{\ideal{m}}\subseteq\ideal{m}^2.\]
\end{theorem}

\begin{proof}
Since a simplicial complex of the form $\Delta_{d,n}$ is symmetric in the sense that all vertices are exactly identical except for their name, if one variable $x_i$ is in $*\core{\ideal{m}}$, all variables are in $*\core{\ideal{m}}$, which forces every minimal $*$-reduction of $\ideal{m}$ to be exactly equal to $\ideal{m}$.  Since $d < n$, this is an impossibility. Thus $*\core{\ideal{m}}$ contains no degree one monomials.  But $*\core{\ideal{m}}$ is generated by monomials and the smallest possible remaining monomial members of $*\core{\ideal{m}}$ are the degree two monomials.  Thus $*\core{\ideal{m}}\subseteq\ideal{m}^2.$
\end{proof}

\begin{theorem}
\label{bound}
If $\Delta$ is a simplicial complex of dimension $d-1$ on $n$ vertices, then for $\mathfrak{m}$ of $k[\Delta]$,\[\ideal{m}^{d+1}\subseteq *\core{\ideal{m}}\subseteq\ideal{m}^{2}.\]
\end{theorem}

\begin{proof}
Let $I$ represent an arbitrary minimal $*$-reduction of $\ideal{m}$ in $k[\Delta_{d,n}]$.  Then $*\core\ideal{m}=\cap I$.  If $f:k[\Delta_{d,n}]\twoheadrightarrow k[\Delta]$ is the natural surjection between the Stanley-Reisner ring associated to $\Delta_{d,n}$ and the Stanley-Reisner ring associated to $\Delta$, a proper sub simplicial complex of the same dimension as $\Delta_{d,n}$, then $f(I)$ is a $*$-reduction of $f(\ideal{m})=\ideal{n}$, the maximal ideal generated by the variables of $k[\Delta]$.  If $J$ is an arbitrary minimal $*$-reduction of $\ideal{n}$, \[*\core\ideal{n}=\cap J\subseteq\cap f(I)\subseteq f(\ideal{m}^2)=\ideal{n}^2\]

Preserving the lower bound only requires a relaxing of conditions in Lemma \ref{lemma}, indeed the lemma is true for all $\Delta$. All monomials corresponding to the length $d+1$ partition of $d+1$ are zero in $k[\Delta]$, so they are, by default, in every ideal of $k[\Delta]$, including $*\core\ideal{n}$. If we want to show the monomial $x_1^{\alpha_1}\cdots x_s^{\alpha_s}$ such that $\alpha_1+\cdots+\alpha_t=d+1$ and $\alpha_i\geq\alpha_{i+1}$ is in $*\core{\ideal{n}}$, choose a minimal prime $P=(y_1,\ldots,y_t)$ of $k[\Delta]$ such that $\{x_1,\ldots,x_s\}\cap\{y_1,\ldots,y_t\}=\varnothing$. If no such ideal exists, $x_1^{\alpha_1}\cdots x_s^{\alpha_s}=0$. Diagonalize and reduce the linear variables of the generators of an arbitrary minimal $*$-reduction $J$ and then manipulate the generators so that the nonlinear terms are all divisible by a variable in $P$. One of the generators of $J$ will be of the form $x_1+g$ where $g$ is the linear terms in the prime $P$ plus non linear terms divisible by variables in this prime. Then if all the terms of $x_1^{\alpha_1-1}\cdots x_s^{\alpha_s}\cdot g$ are in $*\core{\ideal{n}}$, then so is $x_1^{\alpha_1}\cdots x_s^{\alpha_s}$. All these terms are in $*\core{\ideal{n}}$ by the same induction as in Theorem \ref{lower}.
\end{proof}

\section{Special Cases}
\label{cases}

Using the techniques and machinery built in previous sections, we can calculate $*\core{\ideal{m}}$ for many classes of Stanley-Reisner rings without intersecting every minimal $*$-reduction of $\ideal{m}$.  What these calculations suggest is that the structure of the simplicial complex plays a significant role in determining the $*\core$ of $\ideal{m}$.  We will first look at what happens when the simplicial complex consists of disjoint components. In the simplest form the disjoint components are all independently simplices, such as the example $k[x,y]/(xy)$.

\begin{prop} 
\label{p1}
If $\Delta$ is a disjoint union of two or more simplices, then $*\core{\ideal{m}}=\ideal{m}^{2}$.
\end{prop}
\begin{proof}
For simplices $\Delta_{i}$ with $1\leq i\leq r$, let $$\Delta=\bigcup_{i=1}^{r}\Delta_i$$ such that $\Delta_i\cap\Delta_{j}=\emptyset$ for any choice of $i$ and $j$, $i\neq j$.  Let the largest element of $\Delta_{i}$ be the set $\{v_{i,1},\ldots,v_{i,n_{i}}\}$.  Then $$k[\Delta]=k[x_{1,1},\ldots,x_{1,n_1},x_{2,1},\ldots,x_{2,n_{2}},\ldots,x_{r,1},\ldots,x_{r,n_r}]/I_{\Delta}$$ where $I_{\Delta}$ is generated by all degree two monomials $x_{i,s}x_{j,t}$, with $i\neq j$, $1\leq s\leq n_i$, and $1\leq t \leq n_{j}$.

Let $X$ be the set of all variables.  The ring $k[\Delta]$ has $r$ minimal primes $P_{1},\ldots,P_{r}$ such that $P_{i}=(X-\{x_{i,1_{i}},\ldots,x_{i,n_i}\})$.  The annihilator of $P_{i}$ is the ideal $(x_{i,1_i},\ldots, x_{i,n_i})$.  Therefore the test ideal $\tau=\ideal{m}$ and $$\tau\cdot \ideal{m}=\ideal{m}^{2}\subseteq *\core{\ideal{m}}.$$ Since we have shown in Theorem \ref{bound} that an upper bound for $*\core{\ideal{m}}$ is always $\ideal{m}^{2}$, $*\core{\ideal{m}}=\ideal{m}^2$.
\end{proof}

As a corollary to this theorem, the $*\core{\ideal{m}}$ in $k[\Delta]$ when $\dim\Delta=0$ is $\ideal{m}^2$.  This fact can also be inferred from the previously established bounds for the $*\core{\ideal{m}}$ because in any one dimensional ring, $\ideal{m}^{d+1}=\ideal{m}^2$.

We will now relax the condition that that the disjointed pieces of the simplicial complex are simplices. We will first look at rings where the complex is a disjoint union between a proper simplicial complex and a simplex such as the ring $k[w,x,y,z]/(xy,wz,xz,yz)$.  We will immediately follow with what happens to the $*\core{\ideal{m}}$ when the two disjoint pieces are both proper.

\begin{prop}
\label{p2}
Let $\Delta=\union$ where $\Delta_1$ is a proper simplicial complex on the variable set $\{x_1,\ldots,x_n\}$ and $\Delta_2$ is a simplex on the variable set $\{y_1,\ldots,y_m\}$ that is disjoint from $\Delta_1$.  Let $\ideal{m}$ be an ideal of $k[\Delta]$.  Then
\[*\core{\ideal{m}}=\varphi^{-1}(*\core{\ideal{m}_1})+(y_{1},\ldots,y_m)^{2},\]
where $\varphi$ is the natural surjection from $k[\Delta]$ to $k[\Delta_{1}]$ and $\ideal{m}_1$ is the maximal ideal generated by the images of the variables in $k[\Delta_1]$.
\end{prop}

\begin{proof}
Let $P_{1},\ldots,P_{s}$ be the minimal primes of the ring $k[\Delta_1]$.  Then the ideals $Q_{i}=P_{i}k[\Delta]+(y_{1},\ldots,y_{m})$ are minimal primes of $k[\Delta]$ for $1\leq i\leq s$.  In addition to the ideals $Q_{i}$, the ideal $(x_{1},\ldots,x_{n})$ is the only other minimal prime of $k[\Delta]$.  If we compute the test ideal $\tau$, we see \[\ann{(Q_{i})}=\ann{(P_{i})}\cdot k[\Delta]\cap (x_1,\ldots,x_{n})=\ann{(P_{i})}\cdot k[\Delta]\] and \[\ann{(x_{1},\ldots,x_{n})}=(y_{1},\ldots,y_{m})\] so \[\tau=\sum\limits_{i=1}^s\ann{(P_i)}\cdot k[\Delta]+(y_1,\ldots,y_{m}).\]The lower bound we get for the $*\core{\ideal{m}}$ by computing $\tau\cdot\ideal{m}$ is \[\tau\cdot\ideal{m}=\sum\limits_{i=1}^s\ann{(P_i)}\cdot k[\Delta]\cdot(x_{1},\ldots,x_{n})+(y_1,\ldots,y_{m})^2\] which shows that inclusion of $\Delta_2$ in the simplicial complex $\Delta$ results in the inclusion of the degree two monomials in the variables $y_1,\ldots,y_m$ as generators of $*\core{\ideal{m}}$.  Since any product $x_iy_j$ is in $I_{\Delta}$, we need only determine which monomials in the variables $x_{1},\ldots,x_{n}$ are generators of the $*\core{\ideal{m}}$.

For $d=\dim{k[\Delta]}$, let $J=(f_1,\ldots, f_d)$ be a linearly generated minimal $*$-reduction of $\ideal{m}$ in $k[\Delta]$.  Each generator $f_i$ of $J$ can be written $f_i=g_i+h_i$ where $g_i$ is a linear polynomial in the variables $x_{1},\ldots,x_{n}$ and $h_i$ is a linear polynomial in the variables $y_{1},\ldots,y_m$.  Let $\varphi$ be the natural surjection from $k[\Delta]$ to $k[\Delta_1]$ and let $\ideal{m}_1$ be the maximal ideal of $k[\Delta_1]$ generated by the images of the variables.  As we have seen previously, 
\[\ideal{m}_1=\varphi(\ideal{m})=\varphi(J^*)\subseteq \varphi(J)^*\subseteq\ideal{m}_1,\] 
so the ideal $\varphi(J)=(g_1,\ldots,g_d)$ is a $*$-reduction of $\ideal{m_1}$.  Let $\alpha$ be a monomial in $*\core{\ideal{m}_1}$, then for some $\beta_1,\ldots,\beta_d$ in $k[\Delta_1]$, $\alpha=\beta_{1}g_1+\cdots +\beta_dg_d$.  If $\overline{\beta}_i$ represents the element in the inverse image of $\beta_i$ in only the variables $x_1,\ldots,x_n$, then in the ideal $J$, 
\[\overline{\beta}_1f_1+\cdots +\overline{\beta}_df_d=\overline{\beta}_1g_1+\cdots+\overline{\beta}_dg_d\]
is equal to the element of the preimage of $\alpha$ that has no terms in the variables $y_1,\ldots,y_m$.  The only such element that exists is the monomial $\alpha$ itself in $k[\Delta]$.  Thus every monomial in $*\core{\ideal{m}_1}$ is in $*\core{\ideal{m}}$.

The last thing we must show is that if $\alpha$ is a monomial in the variables $x_1,\ldots,x_n$ and $\alpha\in *\core{\ideal{m}}$, then $\alpha$ is also in $*\core{\ideal{m}_1}$. Let $d_1=\dim k[\Delta_1]$ and let $J=(g_1,\ldots,g_{d_1})$ be a minimal $*$-reduction of $\ideal{m}_1$.  One of the following two things is true about $d_1$:  $d_1< d=m$ or $d_1=d$.  To extend $J$ to a minimal $*$-reduction of $\ideal{m}$, it is important to note that $J\cdot k[\Delta]$ is such that $J\cdot k[\Delta]/Q_j=\ideal{m}/Q_j$ for all $j$. So $J\cdot k[\Delta]$ meets all the criteria to be a $*$-reduction of $\ideal{m}$ except for the requirement that $J\cdot k[\Delta]/(x_1,\ldots,x_n)=(y_1,\ldots,y_m)$.  In both cases, we can extend $J$ to a minimal $*$-reduction $J'=(f_1,\ldots,f_d)$ of $\ideal{m}$ by adding $y_i$, for $1\leq i\leq m$, to the $i$th generator of $J$, and using 0 as for all possible generators $g_{d_1+1},\ldots, g_{d}=g_{m}$ if $d_1 < d$.  Let $\alpha$ be a monomial in $*\core{\ideal{m}}$ consisting of only variables $x_1,\ldots,x_n$, then there exist $\beta_{1},\ldots,\beta_{d}$ in $k[\Delta]$ such that $\alpha=\beta_1f_1\cdots+\beta_df_d$.  Each $\beta_i$ can be written $\beta_i=b_i+b_i'$ such that the $b_i$ are the terms of $\beta_i$ in the variables $x_1,\ldots,x_n$ and the $b_i'$ are the terms in variables $y_1,\ldots,y_m$.  Then \[\alpha=\beta_1f_1+\cdots+\beta_df_d=b_1g_1+\cdots+b_1g_{d_1}\in J.\] Thus $\alpha$ is in all minimal $*$-reductions of $\ideal{m}_1$ and therefore, is in $*\core{\ideal{m}_1}$.
\end{proof}

\begin{prop}
\label{p3}
Let $\Delta=\union$ be the disjoint union of two proper simplicial complexes $\Delta_1$ and $\Delta_2$.  Let $\ideal{m}_1$ be the maximal ideal generated by the images of the variables in $k[\Delta_1]$ and let $\ideal{m}_2$ be defined analogously for $k[\Delta_2]$.  Then \[*\core{\ideal{m}}=(*\core{\ideal{m}_1})\cdot k[\Delta]+(*\core{\ideal{m}_2})\cdot k[\Delta].\]
\end{prop}

\begin{proof}
Let $k[\Delta_1]=k[x_1,\ldots,x_n]/I_{\Delta_1}$ and $k[\Delta_2]=k[y_1,\ldots,y_m]/I_{\Delta_2}$ and let $d_1=\dim k[\Delta_1]$ and $d_2=\dim k[\Delta_2]$. Without loss of generality suppose $d=d_1\geq d_2$. Let $P_1,\ldots, P_s$ be the minimal primes of $k[\Delta_1]$ and $Q_1,\ldots,Q_t$ be the minimal primes of $k[\Delta_2]$.  The minimal primes of $k[\Delta]$ are therefore $P_{i}k[\Delta]+(y_1,\ldots,y_m)$ for $1\leq i\leq s$ and $Q_{j}k[\Delta]+(x_1,\ldots,x_n)$ for $1\leq j\leq t$.  If $I_1=(g_1,\ldots,g_d)$ is a minimal $*$-reduction of $\ideal{m}_1$ in $k[\Delta_1]$ and $I_2=(h_1,\ldots, h_{d_2})$ is a minimal $*$-reduction of $\ideal{m}_2$ in $k[\Delta_2]$, we can make a minimal $*$-reduction in $k[\Delta]$ the following way: let $\overline{g_i}$ be the preimage of $g_i$ in the natural surjection $\varphi_1:k[\Delta]\twoheadrightarrow k[\Delta_1]$ that has no additional monomials as terms and let $\overline{h_j}$ be defined analogously for $h_j$ across the surjection $\varphi_2:k[\Delta]\twoheadrightarrow k[\Delta_2]$.  Define polyonomials $f_1,\ldots, f_d$ of $k[\Delta]$ the following way:
\[f_i=\begin{cases}
\overline{g_i}+\overline{h_i} & \text{for }1\leq i\leq d_2\\
\overline{g_i} & \text{for }d_2+1<i\leq d.
\end{cases}\]Then $I=(f_1,\ldots,f_d)$ is a minimal $*$-reduction of $\ideal{m}$ in $k[\Delta]$.

Let $\alpha$ be a nonzero monomial in $*\core{\ideal{m}}$ and let $I$ be a minimal $*$-reduction of $\ideal{m}$ of the type defined above.  Either $\alpha$ is a product of the $x_i$ variables or it is a product of the $y_j$ variables.  Suppose the former.  Then for the ideal $I=(f_1,\ldots,f_d)$ of $k[\Delta]$, there exist $a_1,\ldots, a_d$ in $k[\Delta]$ such that \[\alpha=a_1f_1+\cdots +a_df_d.\] Since any product of an $x$ and a $y$ is 0, we can assume the individual monomials of all the $f$ and $g$ polynomials are of one of the two types of variables.  Since $\alpha$ is all $x$ variables, the sum of the $y$ variables in $a_1f_1+\cdots a_df_d$ is 0.  Specifically if $\overline{g_i}$ is the part of $f_i$ with $x$ variables and $b_i$ is the part of $a_i$ with $x$ variables, \[\alpha=b_1\overline{g_1}\cdots+b_d\overline{g_d}\] which shows $\alpha$ to be in the ideal $I'=(\overline{g_1},\dots, \overline{g_d})$ of $k[\Delta]$ and the image of $\alpha$ in $k[\Delta_1]$ is in the minimal $*$-reduction $I_1=(g_1,\ldots,g_d)$ of $k[\Delta_1]$. Since this works for all such $I_1$, $\alpha$ is a monomial in $*\core{\ideal{m}_1}$.  Similarly, if $\beta\in *\core{m}$ is a monomial in only the $y$ variables, the image of $\varphi_2(\beta)\in *\core{\ideal{m}_2}$.  Thus \[*\core{\ideal{m}}\subseteq (*\core{\ideal{m_1}})\cdot k[\Delta]+(*\core{\ideal{m}_2})\cdot k[\Delta].\]

Let $I=(f_1,\ldots,f_d)$ be a minimal $*$-reduction of $\ideal{m}$ in $k[\Delta]$. Then $\varphi_1(I)$ is a minimal $*$-reduction of $\ideal{m}_1$ in $k[\Delta_1]$ and $\varphi_2(I)$ is a $*$-reduction of $\ideal{m}_2$ in $k[\Delta_2]$.  Let $\varphi_1(f_i)=g_i$ and $\varphi_2(f_i)=h_i$ and let $\alpha$ be a monomial in $*\core{\ideal{m}_1}$.  Then there exist polynomial $a_1,\ldots, a_d$ in $k[\Delta_1]$ such that $\alpha=a_1g_1+\cdots +a_dg_d$. Let $\overline{a_i}$ be the preimage of $a_i$ in $k[\Delta]$ with no additional monomial terms. Then $\overline{a_1}f_1+\cdots+\overline{a_d}f_d$ is the monomial preimage of $\alpha$ in $k[\Delta]$.  Thus $(*\core{\ideal{m}_1})\cdot k[\Delta]\subseteq *\core{\ideal{m}}$.  Similarly, $(*\core{\ideal{m}_2})\cdot k[\Delta]\subseteq *\core{\ideal{m}}$.  Hence \[(*\core{\ideal{m}_1})\cdot k[\Delta]+(*\core{\ideal{m}_2})\cdot k[\Delta]\subseteq *\core{\ideal{m}},\] which means \[*\core{\ideal{m}}=(*\core{\ideal{m}_1})\cdot k[\Delta]+(*\core{\ideal{m}_2})\cdot k[\Delta].\]
\end{proof}

We can infer from propositions \ref{p1}, \ref{p2}, and \ref{p3} that we need only explicitly calculate the $*\core{\ideal{m}}$ in when a simplicial complex connected. For the rest of this paper, we will assume that no simplicial complexes contain disjoint elements unless otherwise stated.

\begin{prop}
\label{p4}
Let $\Delta$ be a simplicial complex with exactly two distinct facets.  Then $*\core{\ideal{m}}=\ideal{m}^{2}$.
\end{prop}
\begin{proof}
Let $\Delta=\Delta_1\cup\Delta_2$ with the vertices of $\Delta_1$ associated to the variable set $\{x_1,\ldots,x_n,z_1,\ldots,z_r\}$ and the vertices of $\Delta_2$ associated to the variable set $\{y_{1},\ldots,y_m,z_1,\ldots,z_r\}$ where $\{z_1,\ldots,z_r\}$ is the set of variables associated to the face $\Delta_1\cap\Delta_2$.  The defining ideal of the simplicial complex is $I_{\Delta}=(\{x_{i}y_{j} : 1\leq i\leq n\text{ and }1 \leq j\leq m\})$.  The ring $k[\Delta]$ has two minimal primes: $P=(y_1,\ldots,y_m)$ and $Q=(x_1,\ldots,x_n)$.  Therefore, the annihilator of $P$ is $Q$ and the annihilator of $Q$ is $P$.

The test ideal $\tau$ is generated by the generators of the two annihilators of the minimal primes i.e. $\tau=(x_{1},\ldots,x_{n},y_{1},\ldots,y_{m})$.  Thus $$\tau\cdot\ideal{m}= (P+Q)\cdot\ideal{m}=P^2+P\cdot Q+Q^2+P\cdot(z_1,\ldots,z_r)+Q\cdot(z_1,\ldots,z_r)$$ is contained in $*\core\ideal{m}$.  To show that $\ideal{m}^2\subseteq *\core{\ideal{m}}$, we need only show that $(z_1,\ldots,z_r)^2\subseteq *\core\ideal{m}$.

Without loss of generality, suppose that $\dim\Delta_1\geq\dim\Delta_2$.  Then the dimension of $\Delta$ is $n+r-1$.  Then for any minimal $*$-reduction $I$ of $\ideal{m}$, $$I=(x_1+g_1,\ldots,x_n+g_n,z_1+h_1,\ldots, z_r+h_r)$$ where each $g$ and each $h$ are polynomials with linear terms in $y_1,\ldots,y_m$ and non linear terms divisible by at least one $y_{i}$.  Since $P\cdot (z_1,\ldots,z_r)\subseteq *\core{\ideal{m}}$, For any choice of $i,j$ between $1$ and $r$, $z_ih_j\in I$, so $$z_iz_j=z_iz_j+z_ih_j-z_ih_j=z_i(z_j+h_j)-z_ih_j\in I.$$  Thus $(z_1,\ldots,z_r)^2\subseteq *\core{\ideal{m}}$ and $*\core{\ideal{m}}=\ideal{m}^2$.
\end{proof}

At this point we must introduce a new notion which we will call a linear $*\core$ of an ideal which we will abbreviated $l*\core$. We define the linear $*\core$ of $I$ to be the intersection of all linear minimal $*$-reductions of $I$. The reason we are introducing this notion is that the non linear parts of the generators of a $*$-reduction and a level of complexity that we can avoid using only linear generators. What should be immediately clear is that $*\core I\subseteq l*\core I$.  If in any case the reverse containment is true, we can discard the non linearly generated reductions when calculating $*\core I$.

For this discussion, let $R=k[\Delta]$ be a Stanley-Reisner ring of dimension $d$. Let $J=(f_1,\ldots,f_d)$ be linearly generated and $I=(f_1+g_1,\ldots, f_d+g_d)$ where $g_i$ has no linear terms such that $I^*=J^*=\ideal{m}$.  It is important to note here, that the linear parts of the generators of a minimal $*$-reduction themselves generate a minimal $*$-reduction.  For any minimal prime $P=(x_{m+1},\ldots,x_n)$, we can rewrite the generators of $J$ and $I$ to be \[J=(x_1+f_1',x_2+f_2',\ldots,x_m+f_m',\ldots,f_d')\]
\[I=(x_1+f_1'+g_1',x_2+f_2'+f_2',\ldots,x_m+f_m'+g_m',\ldots,f_d'+g_d')\]and we will rename $x_i+f_i'$ for $1\leq i\leq m$ and $f_i'$ for $m+1\leq j\leq d$ to be $h_i$. If $a$ is a monomial of degree $q$ in $l*\core{\ideal{m}}$ then $a=b_1+\cdots+b_dh_d$ for $b_1,\ldots, b_d$ which are individually either 0 or homogenous of degree $q-1$. If we carry the same $b_i$ over to $I$, we get that \[b_1(h_1+g_1')+\cdots+b_d(h_d+g_d')=a+b_1g_1'+\cdots+b_dg_d'\] where each $b_ig_i'$ is a polynomial with terms of degree $q+1$ or larger. With out knowing much specifically about the $b_ig_i'$ we get the following lemma and two corollaries:
\begin{lemma}
If $I$ contains all monomials of degre $q+1$, then $I$ also contains all monomials of degree $q$ that are in $J$.
\end{lemma}
\begin{corollary}
\label{cor}
All monomials of degree $d$ that are in $I$ are also in $J$.
\end{corollary}
\begin{proof}
This is a direct consequence of $*\core\ideal{m}$ being bound below by $m^{d+1}$ and the lemma.
\end{proof}
\begin{corollary}
\label{cor2}
If $\dim k[\Delta]\leq 2$, then $*\core\ideal{m}=l*\core\ideal{m}$.
\end{corollary}

\begin{proof}
We know the case for $\dim k[\Delta]=1$ already, and when $\dim k[\Delta]=2$, we have all the degree 3 monomials and we get the degree 2 monomials we need from Corollary \ref{cor}.
\end{proof}

What Corollary \ref{cor2} is saying is that we can compute the $*\core \ideal{m}$ of any simple graph by intersecting only the linearly generated minimal $*$-reductions of $\ideal{m}$.

\begin{prop}
\label{pcycle}
Let the simplicial complex $\Delta$ be a cycle graph.  Then $*\core{\ideal{m}}=\ideal{m}^3$.
\end{prop}

\begin{proof}
Let $V=\{v_1,\ldots,v_n\}$ for $n\geq 3$ be the vertex set of the one dimensional simplicial complex \[\Delta=\{\{v_i\}:1\leq i\leq n\}\cup\{\{v_i,v_{i+1}\}:1\leq i\leq n-1\}\cup\{\{v_1,v_n\}\}\] and let $k[\Delta]=k[x_1,\ldots,x_n]/I_{\Delta}$ be the Stanley-Reisner ring associated to $\Delta$ and $I_{\Delta}$ contains all square free degree two monomials not in $\Delta$.  In all cases, we know that $\ideal{m}^3\subseteq *\core{\ideal{m}}$ by Theorem \ref{bound}.  We will show that $*\core{\ideal{m}}$ contains no degree two monomials.  To do this, we rely on the obvious symmetry of a cycle graph and the contrapositive to Lemma \ref{lemma}.

Case 1: $n$ is odd. Let $I\subset k[\Delta]$ be the ideal generated by the linear polynomials $f_1$ and $f_2$ such that \[f_1=x_1+x_3+\cdots +x_{n-2}+x_n\]\[f_2=x_2+x_4+\cdots +x_{n-1}+x_n\] i.e. $f_1$ is the sum of the odd number variables and $f_2$ is the sum of the even number variables plus $x_n$. By Theorem \ref{inclusion}, $I^*=\ideal{m}$.  Then by definition, $*\core{\ideal{m}}\subseteq I$. We will show $x_1^2\notin I$ and is therefore not in $*\core{\ideal{m}}$.

Suppose there exist polynomials $g_1$ and $g_2$ in $k[\Delta]$ such that $x_1^2=g_1f_1+g_2f_2$.  We can suppose $g_1$ and $g_2$ are linear because all homogeneous degree two polynomials in $I$ will come from products of linear polynomials.  Therefore, let \[g_1=a_1x_1+\cdots+a_nx_n\]\[g_2=b_1x_1+\cdots+b_nx_n\] and we will find values for the coefficients in these two polynomials.  In the product $f_1g_1+f_2g_2$, we get the following set of linear equations to help find the coefficients of $g_1$ and $g_2$: 
\begin{align*}
x_1^2 &: a_1=1\\
x_i^2 &: \begin{cases} 
      a_i=0 & i\text{ is odd and } i\neq 1, n\\
      b_i=0 & i\text{ is even} 
   \end{cases}
\\
x_n^2 &: a_n+b_n=0\\
x_ix_{i+1} &: \begin{cases}
a_{i}+b_{i-1}=0 & i\text{ is even}\\
a_{i-1}+b_{i}=0 & i\text{ is odd, }1<i<n
\end{cases}
\\
x_{n-1}x_n &: a_{n-1}+b_{n-1}+b_n=0\Rightarrow a_{n-1}+b_n=0\\
x_1x_n &: a_1+a_n +b_1=0\Rightarrow a_n+b_1=-1
\end{align*}
The last two equations we change because we know $a_1=1$ and $b_{n-1}=0$.  The last $n+1$ equations listed represent a linear system in $n+1$ variables.  The system of equations is inconsistent, which implies that no such coefficients exist.  Thus, $x_1^2\notin I$ and therefore $x_1^2\notin *\core{\ideal{m}}$.  Because of symmetry of the graph, $x_i^2\notin *\core{\ideal{m}}$ for $1\leq i\leq n$.  By the contrapositive of Lemma \ref{lemma}, this implies that not all monomials of the form $x_ix_j$, $i\neq j$ are in $*\core{\ideal{m}}$.  This can only mean the nonzero monomials, and by symmetry, we can say that the nonzero monomials of the form $x_ix_j$, $i\neq j$ are not in $*\core{\ideal{m}}$.  Thus as ideals of $k[\Delta]$, $*\core{\ideal{m}}\subseteq \ideal{m}^3$.

Case 2: $n$ is even.  The proof of this similar to the case when $n$ is odd. Number the variables of the ring in order around the cycle. Let $I$ be the ideal generated by $f_1$ and $f_2$ such that \[f_1=x_1+x_3+\cdots +x_{n-1}+x_n\]\[f_2=x_2+x_4+\cdots +x_{n-2}+x_n.\]Similar to before, $f_1$ is the sum of the odd numbered variables plus $x_n$ and $f_2$ is the sum of the even numbered variables. This ideal is a minimal $*$-reduction of $\ideal{m}$.  It can be shown that $x_1^2$ is not in this ideal, and consequently, there are no nonzero degree two monomials in $*\core{\ideal{m}}$ and $*\core{\ideal{m}}\subseteq \ideal{m}^3$ in both cases.
\end{proof}
%\section{Acknowledgments}
\section{Integral Closure}
\label{ic}

The ideas of $*$-reductions, $*\spread$, and $*\core$ are studied in part because analogous notions exist in relation to integral closure. For integral closure, we study reductions, analytic spread and core. In fact, any $*$-reduction of an ideal $I$ is also a reduction \cite{nme*spread}, and consequently, core$(I)\subset *\core{(I)}$.  In general, the two cores are not equal. For example, let $R=k[x,y]$ and let $I=(x^2,xy,y^2)$ be an ideal of $R$.  Since $R$ is a polynomial ring over a field, all ideals are tightly closed. But $xy\in (x^2,y^2)^-$, the integral closure of $(x^2,y^2)$. Thus $xy\in*\core I$, but not in $\icore I$. In \cite{FoVa-core}, the authors show that core and $*\core$ agree if analytic spread is equal to $*\spread$ for normal local domains of characteristic $p>0$ with infinite perfect residue fields.  For any Stanley-Reisner ring $k[\Delta]$, an analogous result is true for the core and $*\core$ of $\ideal{m}$ in $k[\Delta]$. We define an ideal to be \textit{basic} if it has no reductions other than itself. The following important lemma is from Hays \cite[Example 2.8]{hays} and helps us show the core and $*\core$ of $\ideal{m}$ are equivalent in Stanley-Reisner rings.
\begin{lemma}
Let $k[x_1,\ldots,x_n]$ with $n\geq 2$ be a polynomial ring over a field.  Then the maximal ideal $(x_1,\ldots, x_n)$ is basic.
\end{lemma} 
\begin{theorem}
Let $k[\Delta]$ be a Stanley-Reisner ring and $\ideal{m}$ the maximal ideal of $k[\Delta]$ generated by the images of the variables. Then every reduction of $\ideal{m}$ is a $*$-reduction of $\ideal{m}$.
\end{theorem}
\begin{proof}
Let $I$ be a reduction of $\ideal{m}$.  Then for all minimal primes $P$ in $k[\Delta]$, $I+P/P$ is a reduction of $\ideal{m}/P$. Since $k[\Delta]/P$ is a polynomial ring, $\ideal{m}/P$ is basic and therefore $I+P/P=\ideal{m}/P$. Thus $I$ is also a $*$-reduction of $\ideal{m}$.
\end{proof}

From this theorem, there are analogous $\icore$ results for the specific examples in Propostions \ref{p1}-\ref{p4} and \ref{pcycle}.  For all Stanley-Reisner rings, we get the following two corollaries:\begin{corollary}
The analytic spread of $\ideal{m}$ is equal to $*\spread{\ideal{m}}$ i.e. the analytic spread of $\ideal{m}$ is \[d=\dim k[\Delta]=\dim\Delta+1.\]In particular, all minimal reductions of $\ideal{m}$ have $d$ generators.
\end{corollary}by Theorems \ref{spreaddn} and \ref{subideal} and\begin{corollary}
For the maximal ideal $\ideal{m}$ of $k[\Delta]$, \[\icore{\ideal{m}}=*\core{\ideal{m}}.\]In particular, \[\ideal{m}^{d+1}+\tau\ideal{m}\subseteq \icore{\ideal{m}}\subseteq\ideal{m}^{2}.\]
\end{corollary}by \cite{FoVa-core} and Theorem \ref{bound}.

\section*{Acknowledgments}
The author would like to thank his advisor Neil Epstein for supplying the initial problem this work is based upon, continued support and direction, and conversations both fruitful and not.
\bibliographystyle{amsalpha}
\bibliography{rsch_refs}

\providecommand{\bysame}{\leavevmode\hbox to3em{\hrulefill}\thinspace}
\providecommand{\MR}{\relax\ifhmode\unskip\space\fi MR }
\providecommand{\MRhref}[2]{%
  \href{http://www.ams.org/mathscinet-getitem?mr=#1}{#2}
}
\providecommand{\href}[2]{#2}
\begin{thebibliography}{AHH93}

\bibitem[AHH93]{AHH}
Ian~M. Aberbach, Melvin Hochster, and Craig Huneke, \emph{Localization of tight
  closure and modules of finite phantom projective dimension}, J. Reine Angew.
  Math. \textbf{434} (1993), 67--114.

\bibitem[BH97]{BH}
Winfried Bruns and J{\"u}rgen Herzog, \emph{{Cohen}-{Macaulay} rings}, revised
  ed., Cambridge Studies in Advanced Mathematics, no.~39, Cambridge Univ.
  Press, Cambridge, 1997.

\bibitem[Eps05]{nme*spread}
Neil Epstein, \emph{A tight closure analogue of analytic spread}, Math. Proc.
  Cambridge Philos. Soc. \textbf{139} (2005), no.~2, 371--383.

\bibitem[FV10]{FoVa-core}
Louiza Fouli and Janet~C. Vassilev, \emph{The $cl$-core of an ideal}, Math.
  Proc. Cambridge Philos. Soc. \textbf{149} (2010), no.~2, 247--262.

\bibitem[FVV11]{FoVaVr-*core}
Louiza Fouli, Janet~C. Vassilev, and Adela Vraciu, \emph{A formula for the
  ${*}$-core of an ideal}, Proc. Amer. Math. Soc. \textbf{139} (2011), no.~12,
  4235--4245.

\bibitem[Hay73]{hays}
James~H. Hays, \emph{Reductions of ideals in commutative rings}, Transactions
  of the American Mathematical Society \textbf{177} (1973), 51--63.

\bibitem[HH88]{HHOrigin}
Melvin Hochster and Craig Huneke, \emph{Tightly closed ideals}, Bull. Amer.
  Math. Soc. \textbf{18} (1988), 45--48.

\bibitem[HH90]{HHmain}
\bysame, \emph{Tight closure, invariant theory, and the {Brian\c{c}on}-{Skoda}
  theorem}, J. Amer. Math. Soc. \textbf{3} (1990), no.~1, 31--116.

\bibitem[HH94]{HHbase}
\bysame, \emph{{$F$}-regularity, test elements, and smooth base change}, Trans.
  Amer. Math. Soc. \textbf{346} (1994), no.~1, 1--62.

\bibitem[HH99]{HH-tcz}
\bysame, \emph{Tight closure in equal characteristic zero}, book-length
  preprint, 1999.

\bibitem[Hoc77]{SR-Hoch}
Melvin Hochster, \emph{{Cohen}-{Macaulay} rings, combinatorics, and simplicial
  complexes}, Proceedings of the Second Oklahoma Ring Theory Conference,
  Lecture Notes in Pure and Appl. Math., vol.~26, Dekker, New York, 1977,
  pp.~171--223.

\bibitem[HS06]{HuSw-book}
Craig Huneke and Irena Swanson, \emph{Integral closure of ideals, rings, and
  modules}, London Math. Soc. Lecture Note Ser., vol. 336, Cambridge Univ.
  Press, Cambridge, 2006.

\bibitem[Kun85]{Kunz-book}
Ernst Kunz, \emph{Introduction to commutative algebra and algebraic geometry},
  Birkh\"auser Boston Inc., Boston, MA, 1985.

\bibitem[Rei76]{SR-Reis}
Gerald~Allen Reisner, \emph{Cohen-{M}acaulay quotients of polynomial rings},
  Advances in Mathematics \textbf{21} (1976), 30--49.

\bibitem[Sta75]{SR-Stan}
Richard~P. Stanley, \emph{Cohen-{M}acaulay rings and constructible polytopes},
  Bull. Amer. Math. Soc \textbf{81} (1975), 133--135.

\bibitem[Vas98]{Va-testquot}
Janet~C. Vassilev, \emph{Test ideals in quotients of {$F$}-finite regular local
  rings}, Trans. Amer. Math. Soc. \textbf{350} (1998), no.~10, 4041--4051.

\end{thebibliography}

\end{document}